\documentclass[11pt]{amsart}
\usepackage{amscd}
\def\b1{B(x^*,\g)}
\def\bb{B(x^*,2\g)}
\def\b3{B(x^*,3\g)}
\def\rta{\rightarrow}
\def\ml{\mathcal{C}}
\def\ml1{\mathcal{C}^1}
\def\mlb1{\mathcal{C}_{b}^{1}}

\def\mbn{\mathbb{R}^n}
\def\g{\gamma}
\def\ua{\mathcal{U}_a}

\def\a{\alpha}
\def\l{\lambda}
\def\frk{\frak}               % font for "Fraktur"

\def\Phi{{\frk n}}
\def\pl{\partial}
\def\opn#1#2{\def#1{\operatorname{#2}}} % to make operators
\opn\chara{char} \opn\length{\ell} \opn\pd{pd} \opn\rk{rk}
\opn\projdim{proj\,dim} \opn\injdim{inj\,dim} \opn\rank{rank}
\opn\depth{depth} \opn\grade{grade} \opn\height{height}
\opn\embdim{emb\,dim} \opn\codim{codim}

\opn\Tr{Tr} \opn\bigrank{big\,rank}
\opn\superheight{superheight}\opn\lcm{lcm}
\opn\trdeg{tr\,deg}%\emph{
\opn\reg{reg} \opn\lreg{lreg} \opn\ini{in} \opn\lpd{lpd}
\opn\size{size}
\opn\div{div} \opn\Div{Div} \opn\cl{cl} \opn\Cl{Cl}
\opn\Spec{Spec} \opn\Supp{Supp} \opn\supp{supp} \opn\Sing{Sing}
\opn\Ass{Ass} \opn\Min{Min}
\opn\Ann{Ann} \opn\Rad{Rad} \opn\Soc{Soc}
\opn\Im{Im} \opn\Ker{Ker} \opn\Coker{Coker} \opn\Am{Am}
\opn\Hom{Hom} \opn\Tor{Tor} \opn\Ext{Ext} \opn\End{End}
\opn\Aut{Aut} \opn\id{id}

\opn\nat{nat}
\opn\pff{pf}%   \pf exists already
\opn\Pf{Pf} \opn\GL{GL} \opn\SL{SL} \opn\mod{mod} \opn\ord{ord}
\opn\Gin{Gin} \opn\Hilb{Hilb}\opn\sdepth{sdepth}
\opn\aff{aff} \opn\con{conv} \opn\relint{relint} \opn\st{st}
\opn\lk{lk} \opn\cn{cn} \opn\core{core} \opn\vol{vol}
\opn\link{link} \opn\star{star}
\opn\gr{gr}

\def\pot#1#2{#1[\kern-0.28ex[#2]\kern-0.28ex]}

\opn\dirlim{\underrightarrow{\lim}}
\opn\inivlim{\underleftarrow{\lim}}
\def\Implies{\ifmmode\Longrightarrow \else
        \unskip${}\Longrightarrow{}$\ignorespaces\fi}
\def\implies{\ifmmode\Rightarrow \else
        \unskip${}\Rightarrow{}$\ignorespaces\fi}
\def\iff{\ifmmode\Longleftrightarrow \else
        \unskip${}\Longleftrightarrow{}$\ignorespaces\fi}

\let\:=\colon
\newtheorem{Theorem}{THEOREM}[section]
\newtheorem{lemma}[Theorem]{LEMMA}

\newtheorem{Remark}[Theorem]{Remark}

\let\epsilon\varepsilon
\let\phi=\varphi
\let\kappa=\varkappa
\def\qed{\ifhmode\textqed\fi
      \ifmmode\ifinner\quad\qedsymbol\else\dispqed\fi\fi}
\def\textqed{\unskip\nobreak\penalty50
       \hskip2em\hbox{}\nobreak\hfil\qedsymbol
       \parfillskip=0pt \finalhyphendemerits=0}
\def\dispqed{\rlap{\qquad\qedsymbol}}

\opn\dis{dis}
\def\pnt{{\raise0.5mm\hbox{\large\bf.}}}

\opn\Lex{Lex}

\begin{document}
\textwidth=13cm
\title{Gradient flows with jumps associated with nonlinear Hamilton-Jacobi equations with jumps}

\author{ SAIMA PARVEEN, Constantin Varsan}

\begin{abstract}We analyze gradient flows with jumps generated by a finite set of complete vector fields in involution using some Radon measures $u\in \mathcal{U}_a$ as admissible perturbations. Both the evolution of a bounded gradient flow $\{x^u(t,\l)\in B(x^*,3\g)\subseteq \mbn: \,t\in[0,T],\,\l\in B(x^*,2\g)\}$ and the unique solution $\l=\psi^u(t,x)\in B(x^*,2\g)\subseteq \mbn$ of integral equation $x^u(t,\l)=x\in B(x^*,\g), \,t\in[0,T]$, are described using the corresponding gradient representation associated with flow and Hamilton-jacobi equations.
\end{abstract}

\maketitle
\section{INTRODUCTION}
For a given finite set of complete vector fields $\{g_1,\ldots,g_m\}\subseteq\mathcal{C}^{\infty}(\mbn;\mbn)$ consider the corresponding local flows $\{G_1(t_1)[x],\ldots,G_m(t_m)[x]:\,|t_i|\leq a_i,\,x\in B(x^*,3\g)\leq\mbn,\,1\leq i\leq m\}$ generated by $\{g_1,\ldots,g_m\}$ correspondingly and satisfying

\begin{equation}\label{1.1}
|G_i(t_i)[x]-x|\leq \frac{\g}{2m},\,x\in B(x^*,3\g),\,|t_i|\leq a_i,\,1\leq i\leq m
\end{equation}
for some fixed constants $a_i>0$ and $\g>0$.

Denote by $\mathcal{U}_a$ the set of admissible perturbations consisting of all piecewise right-continuous mappings (of $t\geq0$) $u(t,x):[0,\infty)\times\mbn\rightarrow\bigsqcup=\mathop{\prod}\limits_{i=1}^{m}[-a_i,a_i]$ fulfilling
\begin{equation}
u(0,\l)=0,\,u(t,.)\in\mlb1(\mbn;\mbn)\,\,\hbox{and}\,\, |\pl_\l u_i(t,\l)|\leq K_1,\,t\geq0,\,\l\in\mbn,\,1\leq i\leq m,
\end{equation}
for some fixed constant $K_1>0$.

For each admissible perturbation $u\in\mathcal{U}_a$, we associate a piecewise right-continuous trajectory (for $t\geq0$)
\begin{equation}\label{l.3}
x^u(t,\l)=G\big(u(t,\l)\big)[\l],\,\,t\geq0,\,\l\in B(x^*,2\g),
\end{equation}
where the smooth mapping $G(p)[x]:\bigsqcup\times B(x^*,2\g)\rightarrow B(x^*,3\g)$ is defined by
\begin{equation}
G(p)[x]=G_1(t_1)\circ\ldots\circ G_m(t_m)[x],\,p=(t_1,\ldots,t_m)\in \bigsqcup,\,x\in B(x^*,2\g)
\end{equation}
verifying $G(p)[x]\in B(x^*,3\g)$ (see \eqref{1.1}).

We are going to introduce some nonlinear ODE with jumps fulfilled by the bounded flow $\{x^u(t,\l):\,t\in[0,T],\,\l\in B(x^*,2\g)\}$ defined in \eqref{l.3}, when $u\in \mathcal{U}_a$ has a bounded variation property. In addition, the unique solution $\{\l=\psi(t,x)\in B(x^*,2\g):\,t\in[0,T],\,x\in B(x^*,\g)\}$ of the integral equation
\begin{equation}
x^u(t,\l)=x\in B(x^*,\g),\,t\in[0,T]
\end{equation}
fulfils a quasilinear  Hamilton-Jacobi (H-J) equation on each continuity interval $t\in[t_k,t_{k+1})\subseteq [0,T]$. These result are contained in the last section of this paper(see Theorems \ref{t.1}, \ref{t:2} and \ref{t:3}). In the case that we assume $\{g_1,\ldots,g_m\}\subset \mathcal{C}^{\infty}(\mbn;\mbn)$ are commuting using Lie bracket then the result are more or less contained in \cite{01}.

Here, in this paper, the vector fields $\{g_1,\ldots,g_m\}\subset \mathcal{C}^{\infty}(\mbn;\mbn)$ are supposed to be in involution over reals which lead us to make use of algebraic representation for gradient systems in a finite dimensional Lie algebra (see \cite{01}) without involving a global nonsingularity  or local times. The analysis performed here reveals the meaningful connection between dynamical systems and partial differential equations.
\section{Formulation of problems and some auxiliary results}
Consider a finite set of complete vector fields $g_i\in\mathcal{C}^{\infty}(\mbn;\mbn),\,1\leq i\leq m$, and let $\{G_i(t_i)[x]:\,|t_i|\leq a_i,\,x\in B(x^*,3\g)\subseteq\mbn\}$ be the local flow generated by $g_i$ satisfying
\begin{equation}\label{2.1}
|G_i(t_i)[x]-x|\leq\frac{\g}{2m},\,x\in B(x^*,3\g),\,|t_i|\leq a_i,\,1\leq i\leq m
\end{equation}
for some fixed constants $a_i>0$ and $\g>0$.

Denote by $\mathcal{U}_a$ the set of admissible perturbations consisting of all piecewise right-continuous mappings(of $t\geq0$)
$u(t,\l): [0,\infty)\times\mbn\rightarrow \bigsqcup=\mathop{\prod}\limits_{i=1}^{m}[-a_i,a_i]$ fulfilling
\begin{equation}\label{2.2}
u(0,\l)=0,\,u(t,.)\in\mlb1(\mbn;\mbn)\,\hbox{and}\,|\pl_\l u_i(t,\l)|\leq K_1,\,t\geq 0,\,\l\in\mbn,\,1\leq i\leq m,
\end{equation}
for some fixed constant $K_1>0$. For each admissible perturbation $u\in\ua$, we associate a piecewise right-continuous trajectory (for $t\geq0$)
\begin{equation}\label{2.3}
x^u(t,\l)=G\big(u(t,\l)\big)[\l],\,\,t\geq0,\,\l\in B(x^*,2\g)\,\,.
\end{equation}
Here the gradient smooth mapping $G(p)[x]:\bigsqcup\times B(x^*,2\g)\rta \b3$ is defined as follows
\begin{equation}\label{2.4}
G(p)[x]=G_1(t_1)\circ\ldots\circ G_m(t_m)[x],\,p=(t_1,\ldots,t_m)\in \bigsqcup,\,x\in B(x^*,2\g)\,,
\end{equation}
and satisfies (see (\ref{2.1})) $G(p)[\l]\in\b3$ for any $p\in\bigsqcup$ and $\l\in\bb$.

The flow with jumps represented as in (\ref{2.3}) stands for a gradient flow with jumps and it relies on the smooth mapping defined in (\ref{2.4}) which is the unique solution of an associated integrable gradient system
\begin{equation}\label{2.5}
\left\{
  \begin{array}{ll}
    \pl_{t_1}y=& g_1(y),\,\pl_{t_2}y=Y_2(t;y),\ldots,\pl_{t_m}y=Y_m(t_1,\ldots,t_{m-1};y), \\
    y(0)=&\l\in\bb,\,p=(t_1,\ldots,t_m)\in\bigsqcup,\,y\in\mbn.
  \end{array}
\right.
\end{equation}
We are looking for sufficient condition on $\{g_1,\ldots,g_m\}$ (see \cite{02}) such that the vector fields with parameters given in (\ref{2.5}) can be represented as follows
\begin{equation}\label{2.6}
\{g_1,Y_2(t_1),\ldots,Y_{m}(t_1,\ldots,t_{m-1})\}(y)=\{g_1,\ldots,g_m\}(y)A(p),\,y\in \b3\,,
\end{equation}\label{2.7}
$p=(t_1,\ldots,t_m)\in\bigsqcup$, where the $(m\times m)$ matrix $A(p)$ satisfies
\begin{equation}\label{2.7.5}
A(0)=I_m,\,A(p)=[b_1b_2(t_1)\ldots b_m(t_1,\ldots,t_{m-1})],\,b_j\in\mathcal{C}^\infty\big(\bigsqcup,\mbn\big),\,b_1=\left(
                                                                                                               \begin{array}{c}
                                                                                                                 1 \\
                                                                                                                 0 \\
                                                                                                                 . \\
                                                                                                                 . \\
                                                                                                                 .\\
                                                                                                                 0 \\
                                                                                                               \end{array}
                                                                                                             \right)
\end{equation}
This algebraic representation help us to define each integral
\begin{equation}\label{2.8}
\int_{0}^{t} b_{j}^{i}\big(u_1(s,\l),\ldots,u_{j-1}(s,\l)\big)d_su_j(s,\l)=\a_{ij}^{u}(t,\l),
\end{equation}
$1\leq j\leq m,\,1\leq i\leq m,\,t\in[0,T]$, as a bounded variation function with respect to $t\in[0,T]$, provided we assume that
\begin{equation}\label{2.9}
\hbox{each}\,\,u_i(t,\l),\,t\in [0,T],\,1\leq i\leq m,\,\hbox{has a bounded variation property}.
\end{equation}
In addition, the algebraic representation (\ref{2.6})(see \cite{02}) can be obtained assuming $g_i\in\mathcal{C}^\infty(\mbn;\mbn),\,1\leq i\leq m$, and
\begin{equation}\label{2.10}
\left\{
  \begin{array}{ll}
    &\{g_1(x),\ldots,g_m(x):\,x\in \b3\subseteq\mbn\}\,\hbox{are in involution}\\
 &\hbox{over reals, i.e each Lie
bracket can be written as} \,[g_i,g_j](x)=\mathop{\sum}\limits_{k=1}^{m}\g_{k}^{ij}g_k(x)\\
&\hbox{for}\,x\in\b3,\,\hbox{using}\,\g_{k}^{ij}\in\mathbb{R}\,.
  \end{array}
\right.
\end{equation}
Let $[t_k,t_{k+1})\subset [0,T],\,0\leq k\leq N-1$, be the continuity intervals of $u\in\ua$.\\
\underline{\textbf{Problem $P_1$}}. Under the hypothesis (\ref{2.9}) and (\ref{2.10}), describe the evolution of the gradient flow with jumps in (\ref{2.3}) as follows
\begin{equation}\label{2.11}
\left\{
  \begin{array}{ll}
    d_tx^u(t,\l)&= \mathop{\sum}\limits_{k=1}^{m}g_i\big(x^u(t,\l)\big)d_t\beta_i^u(t,\l),\,t\in[t_k,t_{k+1}),\,0\leq k\leq N-1\,,\\
    x^u(0,\l)&=\l\in\bb,\,\hbox{where}\,\beta_i^u(t,\l) =\mathop{\sum}\limits_{k=1}^{m}\a_{ij}^u(t,\l),\,1\leq i\leq m\,.
 \end{array}
\right.
\end{equation}
Here the matrix $\{\a_{ij}^u(t,\l):\,i,j\in\{1,\ldots,m\},\,t\in[0,T]\}$ of bounded variation and piecewise right-continuous function of $t\in[0,T]$ are defined in (\ref{2.8}).\\
\underline{\textbf{Problem $P_2$}}. Under the hypothesis (\ref{2.9}), (\ref{2.10}) and for $K_1>0$ sufficiently small (see (\ref{2.2})), prove that the integral equations (with respect to $\l\in\bb$ see (\ref{2.3}))
\begin{equation}\label{2.12}
x^u(t,\l)=x\in B(x^*,\g),\,\,\hbox{are reversibly with respect to }\, \l\in B(x^*,2\g)\,.
\end{equation}
The unique bounded variation and piecewise right-continuous solution $\{\l=\psi^u(t,x)\in B(x^*,2\g);\,t\in[0,T]\}$
%%%%%%%%
is first order continuously differentiable of $x\in\hbox{int}B(x^*,\g)$.
\begin{Remark}
One may wonder about the Hamilton-Jacobi equation with jumps satisfied by the unique solution $\{\l=\psi^u(t,x)\in B(x^*,2\g):\,t\in[0,T],\,x\in B(x^*,\g)\}$ found in ($P_2$). This will be presented at the end of the following section. The next Lemma lead us to the solution of the problem ($P_1$).
\end{Remark}
\begin{lemma}
Assume that the hypothesis (\ref{2.9}) and (\ref{2.10}) are fulfilled and consider the gradient flow with jumps $\{x^u(t,\l):t\in[0,T],\l\in B(x^*,2\g)\}$ defined in (\ref{2.3}), where $u\in\ua$ and $T>0$ are fixed arbitrarily. Then there exists an ($m\times m$) matrix composed by bounded variation and piecewise right-continuous functions $\{\a_{ij}^u(t,\l):\,\a_{ij}^u(0,\l)=0,\,1\leq i,j\leq m,\,t\in[0,T],\,\l\in B(x^*,2\g)\}$ (see (\ref{2.8})) such that
\begin{equation}\label{2.13}
\left\{
  \begin{array}{ll}
   & d_t x^u(t,\l)=\mathop{\sum}\limits_{i=1}^{m}g_i\big(x^u(t,\l)\big)d_t \beta_i^u(t,\l),\,t\in[t_k.t_{k+1}),\,0\leq k\leq N-1,\\
   & x^u(0,\l)=\l,\,\hbox{where}\,\,\beta_i^u\mathop{=}\limits^{\hbox{def}}\mathop{\sum}\limits_{j=1}^{m}\a_{ij}^u(t,\l),\,1\leq i\leq m,
  \end{array}
\right.
\end{equation}
\noindent and $[t_k,t_{k+1})\subseteq[0,T],\,0\leq k\leq N-1,$ are the continuity intervals of $u\in\ua$.
\end{lemma}
\begin{proof}
By definition, $x^u(t,\l)\in\b3,\,t\geq0,\,\l\in B(x^*,2\g)$ (see (\ref{2.3})) where $x^u(t,\l)=G(u(t,\l))[\l]$ defined in (\ref{2.4}) fulfils the integrable gradient system given in (\ref{2.5})(see \cite{02}). In addition, using the hypothesis (\ref{2.10}) (see \cite{02}) we may and do represent the vector fields of (\ref{2.5}) as in (\ref{2.6}). As far as $x^u(t,\l)=y_{\l}(u(t,\l)),\,t\in[0,T],\,\l\in B(x^*,2\g)$, where $\{y_\l(p):p\in \bigsqcup\}$ is the unique solution of (\ref{2.5}), we get the conclusion (\ref{2.13}) provided the algebraic representation (\ref{2.6}) and (\ref{2.7.5}) is used. The proof is complete.
\end{proof}
\begin{Remark}
For solving integral equation $x^u(t,\l)=x\in B(x^*,\g)$ (for some fixed $u\in\ua$) using integral representation (\ref{2.3}), we notice that these are equivalent with the following integral equations
\begin{equation}\label{2.14}
\l=H\big(u(t,\l)\big)[x],\,t\in[0,T],\,x\in B(x^*,\g)
\end{equation}
with respect to $ \l\in B(x^*,2\g)$. Here $H(p)[x]=[G(p)]^{-1})(x)$ satisfies
\begin{equation}\label{2.15}
H(p)[x]\mathop{=}\limits^{\hbox{def}}G_m(-t_m)\circ\ldots\circ G_1(-t_1)[x]\in B(x^*,2\g),\,\hbox{for any}\,\,p=(t_1,\ldots,t_m)\in\bigsqcup
\end{equation}
\noindent and $x\in B(x^*,\g)$.

In addition, using the hypothesis (\ref{2.10}) and writing the corresponding integrable gradient system for $y(p;\l)=G(p)[\l]$ (see (\ref{2.5}) and (\ref{2.6})) we get each $\pl_{t_i}(H(p)[x])$ as follows
\begin{equation}\label{2.16}
  \begin{array}{ll}
    &\pl_{t_1}H(p)[x]=-\pl_x\big(H(p)[x]\big)g_1(x),\,\pl_{t_2}H(p)[x]=-\pl_x\big(H(p)[x]\big)Y_2(t_1;x), \\
    &,\ldots,\pl_{t_m}H(p)[x]=-\pl_x\big(H(p)[x]\big)Y_m(t_1,\ldots,t_{m-1};x).
  \end{array}
\end{equation}
\noindent Here a direct computation is applied to the identity $H(p)[G(p)(\l)]=\l$ and write $0=\pl_{t_i}H(p)[x]+\pl_x\big(H(p)[x]\big)Y_i(t_1,\ldots,t_{i-1;x})$ for each $i\in \{1,\ldots,m\}$, where $Y_1(x)=g_1(x)$ and (see (\ref{2.5}) and (\ref{2.6}))
\begin{equation}\label{2.17}
\{g_1(x),Y_2(t_1;x),\ldots,Y_m(t_1,\ldots,t_{m-1};x)\}=\{g_1(x),\ldots,g_m(x)\}A(p),\,p\in\bigsqcup\,.
\end{equation}
\end{Remark}
\noindent Denote $z(p,x)=H(p)[x]$.
\begin{lemma}\label{l:2}
Assume that the hypothesis (\ref{2.10}) is satisfied and define $H(p)[x]=[G(p)]^{-1}(x)=G_m(-t_m)\circ\ldots G_1(-t_1)[x],\,x\in B(x^*,\g),\,p=(t_1,\ldots,t_m)\in \bigsqcup$, where $y=G(p)[\l],\,p\in\bigsqcup,\,\l\in B(x^*,2\g)$, verifies (\ref{2.4}) and is the unique solution of the integrable gradient system (\ref{2.5}) and (\ref{2.6}). Then there exists an $(m\times m)$ analytic matrix $A(p)$ verifying (\ref{2.17}) such that the following system of (H-J) equation is fulfilled
\begin{equation}\label{2.18}\left\{
                              \begin{array}{ll}
                                &\pl_pz(p;x)+\pl_x(z(p;x))\{g_1(x),\ldots,g_m(x)\}A(p)=0,\,p\in \bigsqcup,\,x\in B(x^*,\g) \\
                               & z(0;x)=x
                              \end{array}
                            \right.
\end{equation}
\end{lemma}
\begin{proof}
A direct computation applied to the identity $H(p)[G(p)(\l)]=\l$ lead us to the following system of (H-J) equations (see $z(p;x)=H(p)[x]$)
\begin{equation}\label{2.19}
\pl_{t_i}z(p;x)+\pl_x(z(p;x))Y_i(t_1,\ldots,t_{i-1};x)=0,\,1\leq i\leq m,\,\forall\,\,p\in\bigsqcup,\,x\in B(x^*,\g)\,.
\end{equation}
Here the vector fields with parameters $\{Y_1,\ldots,Y_m\}$ are defined in (\ref{2.5}) and fulfils the algebraic representation given in (\ref{2.6}). Using (\ref{2.6}), we rewrite (\ref{2.19}) as follows
\begin{equation}\label{2.20}
\left\{
  \begin{array}{ll}
 &\pl_p z(p;x)+\pl_x(z(p;x))\{g_1(x),\ldots,g_m(x)\}A(p)=0,\,p\in\bigsqcup,\,x\in B(x^*,\g)\\
   & z(0;x)=x
  \end{array}
\right.
\end{equation}
and the proof is complete.
\end{proof}
\begin{lemma}\label{l:3}
Under the conditions assumed in Lemma \ref{l:2}, define
\begin{equation}\label{2.21}
V^u(t,x;\l)=z\big(u(t,\l);x\big),\,t\geq0,\,\l\in\mbn,\,x\in B(x^*,\l)\,,
\end{equation}
where $u\in\ua$ is fixed and $u(p;x),\,p\in\bigsqcup,\,\,x\in B(x^*,\l)$, satisfies (H-J) equations (\ref{2.18}). Then the ($n\times n$) matrix $M^u(t,x;\l)\mathop{=}\limits^{\hbox{def}}\pl_\l V^u(t,x;\l)$, verifies the following inequality
\begin{equation}\label{2.22}
|M^u(t,x;\l)|\leq C_1C_2K_1,\,t\geq0,\,\l\in\mbn,\,x\in B(x^*,\g)\,,
\end{equation}
where $K_1>0$ is fixed in (\ref{2.2})(see definition of $\ua$) and
\begin{equation}\label{2.23}
C_1\mathop{=}\limits^{\hbox{def}}\hbox{max}\{|\pl_x(z(p;x))g_i(x)|:\,p\in\bigsqcup,\,x\in B(x^*,\g),\,1\leq i\leq m\}\,,
\end{equation}

\begin{equation}\label{2.24}
C_2\mathop{=}\limits^{\hbox{def}}\hbox{max}\{|A(p)|:\,p\in\bigsqcup\}\,(A(p)\,\,\hbox{is given in (\ref{2.17}) and used in (\ref{2.18})}),
\end{equation}
\end{lemma}
\begin{proof}
By hypothesis, the mapping $z(p;x)=H(p)[x]$ defined in Lemma \ref{l:2} fulfils (H-J) equation (\ref{2.18}) and for an arbitrary $u\in\ua$, we get
\begin{equation}\label{2.25}
M^u(t,x;\l)=\pl_pz(u(t,\l);x)\pl_\l u(t,\l),\,t\geq0,\,\l\in\mbn,\,x\in B(x^*,\g).
\end{equation}

Here $u(t,\l)\in\bigsqcup\subseteq\mathbb{R}^m$ and $|\pl_\l u_i(t,\l)|\leq K_1,\,1\leq i\leq m$, for any $t\geq 0,\,\l\in\mbn$ (see definition of $\ua$ in (\ref{2.2})). On the other hand, using (\ref{2.18}) of Lemma \ref{l:2}, the following inequality is valid
\begin{equation}\label{2.26}
|\pl_p z(u(t;\l);x)|\leq C_1C_2,\,t\geq 0,\,\l\in\mbn,\,x\in B(x^*,\g)
\end{equation}
where the constants $C_1,C_2$ are given in (\ref{2.23}),(\ref{2.24}). A direct computation applied to (\ref{2.25}) leads us to
\begin{equation}\label{2.27}
|M^u(t,x;\l)|\leq C_1C_2K_1,\,t\geq0,\,\l\in\mbn,\,x\in B(x^*,\g)
\end{equation}
and the proof is complete.
\end{proof}
\begin{lemma}\label{l:4}
Assume that $u\in \ua$ and $\{g_1,\ldots,g_m\}\subseteq\mathcal{C}^{\infty}(\mbn;\mbn)$ satisfies (\ref{2.9}) and (\ref{2.10}). Consider $z(p;x)=H(p)[x]$ which verifies (H-J) equations (\ref{2.18}) of Lemma \ref{l:2} and define
\begin{equation}\label{2.28}
V^u(t,x;\l)=z(u(t,\l);x),\,t\in[0,T],\,\l\in\mbn,\,x\in B(x^*,\g)\subseteq\mbn.
\end{equation}
Let $\{\a_{ij}^u(t,\l):\,t\in[0,T],\,\l\in\mbn,\,1\leq i,j\leq m\}$ be the $(m\times m)$ matrix given in \eqref{2.8} and define new bounded variation piecewise right-continuous function $\beta_{i}^{u}(t,\l)=\mathop{\sum}\limits_{j=1}^{m}\a_{ij}^u(t,\l),\,t\in[0,T],\,1\leq i\leq m$. Then $\{V^u(t,x;\l):\,t\in[0,T]\}$ is a bounded variation piecewise right-continuous mapping satisfying the following (H-J) equations with jumps
\begin{equation}\label{2.29}
\left\{
  \begin{array}{ll}
   & d_tV^u(t,x;\l)+\pl_x V^u(t,x;\l)[\mathop{\sum}\limits_{j=1}^{m}g_i(x)d_t\beta_i^u(t,\l)]=0\\
    &V^u(0,x;\l)=x,\,t\in[t_k,t_{k+1}),\,x\in\hbox{Int} B(x^*,\g),\,\l\in\mbn,\,0\leq k\leq N-1
  \end{array}
\right.
\end{equation}
where $[t_k,t_{k+1})\subseteq[0,T],\,0\leq k\leq N-1$, are the continuity intervals of $u\in \ua$.
\end{lemma}
\begin{proof}
By hypothesis, the conclusion \eqref{2.18} of Lemma \ref{l:2} is valid. By a direct computation, we get $V^u(0,x;\l)=x$ and
\begin{equation}\label{2.30}
d_tV^u(t,x;\l)=\mathop{\sum}\limits_{j=1}^{m}\pl_{t_i}z(u(t,\l);x)d_tu_i(t,\l),\,\,t\in[t_k,t_{k+1})\,.
\end{equation}
Using \eqref{2.18}, rewrite \eqref{2.30} as follows
\begin{equation}\label{2.31}
d_t V^u(t,x;\l)=-\pl_x V^u(t,x;\l)\{g_1(x),\ldots,g_m(x)\} A(u(t,\l))\left(
                                                                       \begin{array}{c}
                                                                         d_tu_1(t,\l) \\
                                                                         . \\
                                                                         . \\
                                                                         . \\
                                                                         d_tu_m(t,\l) \\
                                                                       \end{array}
                                                                     \right),
\end{equation}
where $A(p)=[b_1,b_2(t_1),\ldots,b_m(t_1,\ldots,t_{m-1})],\,b_j\in\mathcal{C}^{\infty}(\bigsqcup,\mathbb{R}^m),\,t\in[t_k,t_{k+1})$. Using \eqref{2.8}, write
\begin{equation}\label{2.32}
\a_{ij}^u(t,\l)=\int_{0}^{t}b_j^i\big(u_1(s-,\l),\ldots,u_{j-1}(s-,\l)\big)d_su_j(s,\l),\,1\leq i,j \leq m\,,
\end{equation}
for $t\in[0,T],\,\l\in\mbn$. Rewrite \eqref{2.31} (using \eqref{2.32}) and we get conclusion \eqref{2.29}. The proof is complete.
\end{proof}
\begin{Remark}
The (H-J) equations with jumps satisfied by the unique solution of Problem $(P_2)$ are strongly connected with Lemmas \ref{l:2} and \ref{l:4}. On the other hand, the existence of a solution for Problem $(P_2)$ relies on Lemma \ref{l:3} and it will be analyzed in the next Lemma assuming that $K_1>0$ satisfies
\begin{equation}\label{2.33}
C_1C_2K_1=\rho\in[0,\frac{1}{2}]
\end{equation}
\end{Remark}
\begin{lemma}\label{l:5}
Assume that $u\in\ua$ and $\{g_1,\ldots,g_m\}\subseteq\mathcal{C}^{\infty}(\mbn;\mbn)$ fulfil \eqref{2.9}, \eqref{2.10} and \eqref{2.33}. Then there exists a unique bounded variation piecewise right continuous (of $t\in[0,T]$) mapping $\{\l=\psi^u(t,x)\in B(x^*,2\g):\,t\in[0,T],\,x\in B(x^*,\g)\}$ which is first order continuously differentiable of $x\in int\,B(x^*,\g)$, satisfying integral equations
\begin{equation*}
\left\{
  \begin{array}{ll}
    x^u(t-,\psi^u(t-,x))=x\in B(x^*,\g),\,\,t\in[0,T]\,, \\
    \psi^u(t-,x)=V^u(t-,x;\psi(t-,x)),\,\psi^u(t,x)=V^u(t,x;\psi^u(t-,x)),\,t\in[0,T]\,.
  \end{array}
\right.
\end{equation*}
\end{lemma}
\begin{proof}
By hypothesis, the conclusion of Lemma \ref{l:3} is valid for $V^u(t,x;\l)=z(u(t,\l);x),\,t\geq0,\,\l\in\mbn,\,x\in B(x^*,\g)$. Notice that $x^u(t,\l)=x$ can be rewritten as
\begin{equation}\label{2.34}
\l=V^u(t,x;\l),\,\,t\in[0,T],\,x\in B(x^*,\g)\,,
\end{equation}
where the $(n\times n)$ matrix $\pl_\l V^u(t,x;\l)=M^u(t,x;\l)$ fulfils the conclusion \eqref{2.22}, i.e
\begin{equation}\label{2.35}
|M^u(t,x;\l)|\leq C_1C_2K_1,\,\,t\geq0,\,\l\in\mbn,\,x\in B(x^*,\g)\,.
\end{equation}
Assuming that $K_1>0$ is sufficiently small such that
\begin{equation}\label{2.36}
\rho=C_1C_2K_1\in[0,\frac{1}{2}]\,\,(\hbox{see}\,\eqref{2.33})\,,
\end{equation}
then the contraction mapping theorem can be applied for solving integral equations \eqref{2.34}. Construct the convergent sequence $\{\l_k(t,x):\, t\in[0,T],\,x\in B(x^*,\g)\}$ such that
\begin{equation}\label{2.37}
\left\{
  \begin{array}{ll}
    &\l_k(t,x)\in B(x^*,2\g),\,\,\l_0(t,x)=x\\
    &\l_{k+1}(t,x)=V^u(t,x;\l_k(t-,x)),\,k\geq0\,.
  \end{array}
\right.
\end{equation}
Notice that the following estimates are valid
\begin{equation}\label{2.38}
\left\{
  \begin{array}{ll}
    &|\l_{k+1}(t,x)-\l_k(t,x)|\leq \rho^k|\l_1(t-,x)-\l_0(t,x)|,\,k\geq0 \\
    &|\l_{1}(t,x)-\l_0(t,x)|\leq \max\{|G(p)[x]-x|:\,x\in B(x^*,\g)\}\leq\frac{\g}{2}\,,
  \end{array}
\right.
\end{equation}
(see \eqref{2.4}) which lead us to
\begin{equation}\label{2.39}
|\l_k(t,x)-x|\leq\frac{1}{1-\rho}(\frac{\g}{2})\leq\g,\,\,t\in[0,T],\,x\in B(x^*,\g)\,,\,k\geq0\,.
\end{equation}
Combining \eqref{2.38} and \eqref{2.39}, we get $\{\l_k(t,x)\}_{k\geq0}$ fulfils \eqref{2.37} and passing $k\rightarrow\infty$ in \eqref{2.37}, we obtain
\begin{equation}\label{2.40}
\psi^u(t,x)=\mathop{\lim}\limits_{k\rightarrow\infty}\l_k(t,x)\in B(x^*,2\g),\,t\in[0,T],\,x\in B(x^*,\g)
\end{equation}
satisfying integral equations
\begin{equation}\label{2.41}
\left\{
  \begin{array}{ll}
    &\psi^u(t,x)=V^u(t,x;\psi(t-,x))\,, \\
    &\psi^u(t-,x)=V^u(t-,x;\psi(t-,x))\,,\, t\in[0,T],\,x\in B(x^*,\g)\}\,.
  \end{array}
\right.
\end{equation}
The proof is complete.
\end{proof}
\section{Main Theorems}
With the same notations as in Section 1, we reconsider the problems $P_1$ and $P_2$ in more general setting. Consider the local flow $\{G_0(t)[x]:\,|t|\leq T,\,x\in\b3\subseteq \mbn\}$ generated by a complete vector field $g_0\in\mathcal{C}^{\infty}(\mbn;\mbn)$. Assume that
\begin{equation*}
I_1=\left\{
  \begin{array}{ll}
     &\{g_1,\dots,g_m\}\subseteq\mathcal{C}^{\infty}(\mbn;\mbn)\,\hbox{satisfies \eqref{2.10} and }\,[g_0,g_i]=0,\,1\leq i\leq m,\\
    &\hbox{for any}\,\,x\in\b3.
  \end{array}
\right.
\end{equation*}
\begin{equation*}
I_2=\left\{
  \begin{array}{ll}
     &u\in\ua\,\,\hbox{fulfils}\,\eqref{2.9}\,\,\hbox{where}\,\,T>0\,\,\hbox{and}\,\,\bigsqcup=\mathop{\prod}\limits_{1}^{m}[-a_i,a_i]\subseteq\mathbb{R}^n\,,\\
    &\hbox{are fixed such that} \,|G_0(t[x])-x|\leq\frac{\g}{2(m+1)},\,\,|G_i(t_i)[x]-x|\leq\frac{\g}{2(m+1)}\,.
  \end{array}
\right.
\end{equation*}
Let $[t_k,t_{k+1})\subseteq[0,T],\,0\leq k\leq N-1$, be the continuity intervals of $u\in\ua$.\\
\textbf{Problem ($R_1$)}. Under the hypothesis ($I_1$) and $(I_2)$, describe the evolution of the gradient flow with jumps
\begin{equation}\label{3.1}
\{y^u(t,\l)\mathop{=}\limits^{\hbox{def}}G_0(t)\circ G(u(t,\l))[\l]:\,t\in[0,T],\,\l\in B(x^*,2\g)\subseteq\mbn\}
\end{equation}
(see $G(p)=G_1(t_1)\circ\ldots\circ G_m(t_m)$) as a solution of the following system with jumps
\begin{equation}
\left\{
  \begin{array}{ll}\label{3.2}
     d_t y^u(t,\l)=g_0(y^u(t,\l))dt+\mathop{\sum}\limits_{i=1}^{m}g_i(y^u(t,\l))dt\beta_i^u(t,\l),\,y^u(t,\l)\in B(x^*,3\g)\,,\\
    y^u(0,\l)=\l\in B(x^*,2\g),\,t\in[t_k,t_{k+1}],\,\beta_i^u(t,\l)\mathop{=}\limits^{\hbox{def}}\mathop{\sum}\limits_{j=1}^{m}\a_{ij}^u(t,\l)\,,\,1\leq k\leq N-1\,,
  \end{array}
\right.
\end{equation}
where the matrix $\{a_{ij}^u(t,\l):\,1\leq i,j\leq m\}$ is defined in \eqref{2.8}.
\begin{Theorem}\label{t.1}
Assume that $\{g_0,g_1,\ldots,g_m\}\subseteq \mathcal{C}^{\infty}(\mbn;\mbn)$ and $u\in \ua$ fulfil the hypothesis $(I_1)$ and $(I_2)$. Then the gradient flow with jumps $\{y^u(t,\l)\}$ defined in \eqref{3.1} verifies $y^u(t,\l)\in\b3$ and is a solution of the system \eqref{3.2}.
\end{Theorem}
\begin{proof}
By hypothesis, the gradient flow with jumps $\{y^u(t,\l)\}$ can be rewritten $y^u(t,\l)=G(u(t,\l))\circ G_{0}(t)[\l]$ and using $(I_2)$, we get $y^u(t,\l)\in\b3\,,\,t\in[0,T]\,,\,\l\in B(x^*,2\g)$. On the other hand, a direct computation applied to $\{y^u(t,\l)\}$ lead us to the following equations
\begin{equation}\label{3.3}
d_t y^u(t,\l)=g_0\big( y^u(t,\l)\big)dt+\mathop{\sum}\limits_{i=1}^{m}\pl_{t_i}G(u(t,\l))[G_0(t)(\l)]d_tu_i(t,\l),\,t\in[t_k,t_{k+1})
\end{equation}
where $y(p;\mu)\mathop{=}\limits^{\hbox{def}} G(p)[\mu]$ satisfies an integrable gradient system (see \eqref{2.5}),
\begin{equation}\label{3.4}
\pl_{t_1}y=g_1(y),\,\pl_{t_{2}}y=Y_2(t;y),\ldots,\pl_{t_m}=Y_m(t_1,\ldots,t_{m-1};y),
\end{equation}
fulfilling the algebraic representation given in \eqref{2.6} and \eqref{2.7.5}. As a consequence, we may and do rewrite the second term in the right-hand side of \eqref{3.3} as a follows
\begin{equation}\label{3.5}
\mathop{\sum}\limits_{i=1}^{m}\pl_{t_i}G(u(t,\l))[G_0(t)(\l)]d_tu_i(t,\l)=\mathop{\sum}\limits_{i=1}^{m}g_i(y^u(t,\l))d_t\beta_i^u(t,\l),\,t\in[t_k,t_{k+1}),
\end{equation}
$\l\in B(x^*,2\g),\,0\leq k\leq N-1$ where $\beta_i^u(t,\l)\mathop{=}\limits^{\hbox{def}}\mathop{\sum}\limits_{j=1}^{m}\a_{ij}^u(t,\l)$ and the matrix $\{\a_{ij}^u(t,\l):\,1\leq i,j\leq m\}$ is defined in \eqref{2.8}. The proof is complete.
\end{proof}
\begin{Remark}
The evolution of the gradient flow defined in \eqref{3.1} satisfies the same system with jumps given in \eqref{3.2} if the commutative condition $[g_0,g_i](x)=0,\,1\leq i\leq m,\,x\in\b3\subseteq\mbn$, assumed in $(I_1)$ is replaced by
\begin{equation}\label{3.6}
[g_0,g_i](x)=\mathop{\sum}\limits_{k=1}^{m}\gamma_k^ig_k(x),\,x\in\b3,\,1\leq i\leq m\,,
\end{equation}
where $\gamma_k^i\in\mathbb{R}$ are same constants.

The only change which appears is reflected in the algebraic representation corresponding to the gradient integrable system associated with smooth mapping
\begin{equation}\label{3.7}
y(t,p)[\l]=G_0(t)\circ G(p)[\l],\,|t|\leq T,\,p=(t_1,\ldots,t_m)\in\bigsqcup,\,\l\in B(x^*,2\g)\,.
\end{equation}
We get
\begin{equation}\label{3.8}
\pl_ty=g_0(y),\,\pl_{t_1}y=Y_1(t;y),\,\pl_{t_2}y=Y_2(t,t_1;y),\ldots,\pl_{t_m}y=Y_m(t,t_1,\ldots,t_{m-1};y)\,.
\end{equation}
\begin{equation}\label{3.9}
\{g_0,Y_1(t),Y_2(t,t_1),\ldots,Y_m(t,t_1,\ldots,t_{m-1})\}(y)=\{g_0,g_1,\ldots,g_m\}(y) V(t,\mu)
\end{equation}

This time, the $(m+1)\times(m+1)$ analytic matrix $V(t,p)$ has the following structure
\begin{equation}\label{3.10}
\left\{
  \begin{array}{ll}
   V(t,p)=[V_0,V_1(t),V_2(t,t_1),\ldots,V_m(t,t_1,\ldots,t_{m-1})],\,V_j\in\mathbb{R}^{m+1} \\
\\
    V_0^1=\left(
             \begin{array}{c}
               1 \\
               0 \\
               . \\
               . \\
               . \\
0\\
             \end{array}
           \right)\,\hbox{and}\,\,V_i(t,t_1,\ldots,t_{i-1})=\left(
                                                              \begin{array}{c}
                                                                0 \\
                                                                b_i(t,t_1,\ldots,t_{i-1}) \\
                                                              \end{array}
                                                            \right),\,b_i\in\mathbb{R}^m,
  \end{array}
\right.
\end{equation}
$1\leq i\leq m$. A standard computation applied to $y^u(t,\l)=y(t,u(t,\l))[\l]=G_0(t)\circ G(u(t,\l))[\l],\,t\in [0,T]$, leads to
\begin{equation}\label{3.11}
\left\{
  \begin{array}{ll}
    d_ty^u(t,\l)&=\pl_ty(t,u(t,\l))[\l]dt +\mathop{\sum}\limits_{i=1}^{m}\big(\pl_{t_i}y(t,u(t,\l))[\l]\big)d_tu_i(t,\l)\\
    &=g_0(y^u(t,\l))dt+\mathop{\sum}\limits_{i=1}^{m}g_i(y^u(t,\l))d_t\beta_i^u(t,\l),\,t\in[t_k,t_{k+1})
  \end{array}
\right.
\end{equation}
$0\leq k\leq N-1$, where $\beta_i^u(t,\l)=\mathop{\sum}\limits_{j=1}^{m}\a_{ij}^u(t,\l)$ and
$$\a_{ij}^u(t,\l)\mathop{=}\limits^{\hbox{def}}\int_{0}^{t}b_j^i(s,u_1(s-,\l),\ldots,u_{i-1}(s-,\l))d_su_j(s,\l),\,\,1\leq i,j\leq m\,.$$
Here $\{b_1(t),b_2(t,t_1),\ldots,b_m(t,t_1,\ldots,t_{m-1})\}\subseteq \mathbb{R}^m$ are given in \eqref{3.10}.\\
\textbf{Problem ($R_2$)}. Under the hypothesis $(I_1)$ and $(I_2)$ and for $K_1>0$ sufficiently small (see \eqref{2.33} of Lemma \ref{l:5}) prove that the integral equation with respect to $\l\in B(x^*,2\g)$ (see \eqref{3.1})
\begin{equation}\label{3.12}
y^u(t,\l)\mathop{=}\limits^{\hbox{def}} G_0(t)\circ G(u(t,\l))[\l]=x\in B(x^*,\g),\,t\in[0,T]
\end{equation}
has a unique bounded variation and piecewise right continuous solution $\{\l=\psi^u(t,x)\in B(x^*,2\g):\,t\in[0,T]\}$ of \eqref{3.12} which is first order continuously differentiable of $x\in\hbox{int} B(x^*,\g)$.

In addition, the following equations are valid
\begin{equation}\label{3.13}
\left\{
  \begin{array}{ll}
    V^u(t,x;\l)\mathop{=}\limits^{\hbox{def}} H(u(t,\l))[G_0(-t;x)],\,H(p)=[G(p)]^{-1}\,,\,t\in[0,T]\,,\,p\in\bigsqcup\\
   \psi^u(t-,x)=V^u(t-,x;\psi(t-,x))\,,\,\psi^u(t,x)=V^u(t,x,\psi(t-,x))\,,\\
y^u(t-,\psi^u(t-,x))=x\in B(x^*,\g)\,\,,\,\,\hbox{for any}\,\,t\in[0,T]\,.
  \end{array}
\right.
\end{equation}
Define the following two constants
\begin{equation}\label{3.14}
\left\{
  \begin{array}{ll}
    C_1\mathop{=}\limits^{\hbox{def}}\max\{|\pl_y(z(p;y))g_i(y)|:\,p\in\bigsqcup,\,y\in B(x^*,2\g)\}\\
   C_2\mathop{=}\limits^{\hbox{def}}\max\{|A(p)|:\,p\in\bigsqcup\}
  \end{array}
\right.
\end{equation}
where $z(p;y)=H(p)[y]$ and the analytic $(m\times m)$ matrix $A(p)$ is given in \eqref{2.7}. Assume that $K_1>0$ used in the definition of admissible set $\ua$ (see \eqref{2.2}) satisfies
\begin{equation}\label{3.15}
C_1C_2K_1=\rho\in[0,\frac{1}{2}]\,.
\end{equation}
\end{Remark}
\begin{Theorem}\label{t:2}
Assume that $\{g_0,g_1,\ldots,g_m\}\subseteq\mathcal{C}^{\infty}(\mbn,\mbn)$ and $u\in\ua$
 fulfil the hypothesis $(I_1),\,(I_2)$ and \eqref{3.15}. Then there exists a unique bounded variation and piecewise right-continuous solution $\{\l=\psi^u(t,x)\in B(x^*):\, t\in[0,T]\}$ of \eqref{3.12} which is first order continuously differentiable of $x\in\hbox{int} B(x^*,\g)$ such that the integral equations \eqref{3.13} are satisfied. In addition $V(t,x)\mathop{=}\limits^{\hbox{def}} V^u(t,x;\l)$ satisfy the following system of (H-J) equations with jumps
\begin{equation}\label{3.16}
d_t V(t,x)+\pl_x V(t,x)[g_0(x)dt+\mathop{\sum}\limits_{i=1}^{m}g_i(x)d_t\beta_i^u(t,\l)]=0,\,x\in \hbox{int}B(x^*,\g)\,,
\end{equation}
$t\in[t_k,t_{k+1}],\,0\leq k\leq N-1$ where $\beta_i^u(t,\l),\,1\leq i\leq m$, are defined in \eqref{3.2} (see Problem $(R_1)$).
\end{Theorem}
\begin{proof}
Define $\widehat{V}^u(t,y;\l)=H(u(t,\l))[y] $, where $H(p)=[G(p)]^{-1}\,,\,t\in[0,T],\,p\in\bigsqcup$ and $y\in B(x^*,2\g)$. The integral equation \eqref{3.12} can be replaced by the following
\begin{equation}\label{3.17}
\l=\widehat{V}^u(t,y_0(t,x);\l)\mathop{=}\limits^{\hbox{not}} V^u(t,x;\l),
\end{equation}
where $y_0(t,x)=G_0(-t)(x)\in B(x^*,\g_1)$, for any $t\in[0,T]$, $x\in B(x^*,\g)$ and $\g_1=\g(1+\frac{1}{2(m+1)})$ (see $(I_2)$).
It allows us to get the unique solution $\l=\psi^u(t,x)$ as a composition
\begin{equation}\label{3.18}
\psi^u(t,x)=\widehat{\psi}^u(t,y_0(t,x)),\,t\in[0,T],\,\,x\in B(x^*,\g)\,,
\end{equation}
where $\l=\widehat{\psi}^u(t,y),\,t\in[0,T],\,\,y\in B(x^*,\g_1)$, is the unique solution of the following integral equations
\begin{equation}\label{3.19}
\l=\widehat{V}^u(t,y;\l),\,\l\in B(x^*,2\g)\,,y\in B(x^*,\g_1),\,\,t\in[0,T]\,.
\end{equation}
By hypothesis, the mapping $\{\widehat{V}^u(t,y;\l),\,\,t\in[0,T]\,,\,y\in B(x^*,\g_1)\}$ fulfils the hypothesis \eqref{2.9}, \eqref{2.10} and \eqref{2.33} of Lemmas \ref{l:4} and \ref{l:5} for any $\l\in \mbn$. We get the corresponding (H-J) equations (see \eqref{2.29} of Lemma \ref{l:4}).
\begin{equation}\label{3.20}
\left\{
  \begin{array}{ll}
    d_t\widehat{V}^u(t,y;\l)+\pl_y\widehat{V}(t,y;\l)[\mathop{\sum}\limits_{i=1}^mg_i(y)d_t\beta_i^u(t,\l)]=0  \\
     \widehat{V}^u(0,y;\l)=y,\,\,t\in[t_k,t_{k+1}],\,y\in\hbox{int} B(x^*,\g_1),\,\,\l\in\mbn\,,\,0\leq k\leq N-1\,.
  \end{array}
\right.
\end{equation}
In addition, there exists a unique bounded variation piecewise right-continuous (of $t\in[0,T]$) mapping $\{\widehat{\l}=\widehat{\psi}^u(t,y)\in B(x^*,2\g):\,t\in[0,T],\,y\in B(x^*,\g_1)\}$ (see Lemma \ref{l:5})
\begin{equation}\label{3.21}
\left\{
  \begin{array}{ll}
    \widehat{\psi}^u(t-,y)=\widehat{V}^u(t-,y;\widehat{\psi}^u(t-,y)),\,t\in[0,T],\,\,y\in B(x^*,\g_1)  \\
     \widehat{\psi}^u(t,y)=\widehat{V}^u(t,y;\widehat{\psi}^u(t-,y)),\,t\in[0,T],\,\,y\in B(x^*,\g_1)\,,
  \end{array}
\right.
\end{equation}
notice that $\widetilde{\l}=\widetilde{\psi}^u(t,y)$ is first order continuously differentiable of $y\in\hbox{int} B(x^*,\g_1)$ and, using \eqref{3.20} and \eqref{3.21}, we get the corresponding equations satisfied by\\ $\l=\psi^u(t,x)\mathop{=}\limits^{\hbox{def}}\widehat{\psi}^u(t,G_0(-t)(x)),\,\,t\in[0,T]$, as follows
\begin{equation}\label{3.22}
    \psi^u(t-,y)=V^u(t-,x;\widehat{\psi}^u(t-,y)),\psi^u(t,x)=V^u(t,x;\widehat{\psi}^u(t-,y))\,t\in[0,T],
      \end{equation}
$\,\,\,\forall\, x\in B(x^*,\g_1)$ where $\widehat{V}(t,x;\l)=\widehat{V}^u(t,G_0(-t)(x);\l)=H(u(t,\l))[G_0(-t)(x)]$. The equations \eqref{3.22} stands for integral equation \eqref{3.13} and the first conclusion is proved. For the (H-J) equations \eqref{3.16}, we notice that
\begin{equation}\label{3.23}
    d_tV^u(t,x;\l)=d_tV^u(t,G_0(-t)(x);\l)-\pl_y \widehat{V}^u(t,G_0(-t)(x);\l)g_0(G_0(-t)(x))dt,
      \end{equation}
for any $t\in[t_k,t_{k+1}),\,0\leq k\leq N-1$. In addition, using $(I_1)$, we get
\begin{equation}\label{3.24}
\left\{
  \begin{array}{ll}
    \pl_y\widehat{V}^u(t,G_0(-t)(x);\l)g_i(G_0(-t)(x))=\pl_x V^u(t,x;\l)[\pl_x(G_0(-t)(x))]^{-1}g_i(G_0(-t))  \\
     =\pl_x V^u(t,x;\l)g_i(x),\,t\in[0,T],\,\,0\leq i\leq m\,.
  \end{array}
\right.
\end{equation}
Rewrite \eqref{3.23} (using \eqref{3.24} and \eqref{3.20}) we get (H-J) equations \eqref{3.16}. The proof is complete.
\end{proof}
\begin{Theorem}\label{t:3}
Under the hypothesis of Theorem \ref{t:2} and assume that $u\in\ua$ is continuously differentiable on each continuity interval $[t_k,t_{k+1})\subseteq[0,T],\,0\leq k\leq N-1$. Then the unique solution $\{\l=\psi^u(t,x)\in B(x^*,2\g):\,[t_k,t_{k+1}),\,x\in\hbox{int} B(x^*,\g)\}$ of integral equations \eqref{3.12} (see \eqref{3.13} also) satisfies the following (H-J) equations
\begin{equation}\label{3.25}
   \pl_t \psi^u(t,x)+\pl_x\psi^u(t,x)[g_0(x)+\mathop{\sum}\limits_{i=1}^mg_i(x)\pl_t\beta_i^u(t,\psi^u(t,x))]=0\,,
      \end{equation}
$t\in[t_k,t_{k+1}),\,x\in\hbox{int} B(x^*,\g),\,\,0\leq k\leq N-1$. Here $\beta_i^u(t,\l),\,\,1\leq i\leq m$, are defined in Problem $(R_1)$ (see \eqref{3.2}).
\end{Theorem}
\begin{proof}
By hypothesis, the conclusion of Theorem \ref{t:2} are valid and, in particular, the (H-J) equations \eqref{3.16} can be rewritten
\begin{equation}\label{3.26}
   \pl_t V(t,x)+\pl_x V(t,x)[g_0(x)+\mathop{\sum}\limits_{i=1}^mg_i(x)\pl_t\beta_i^u(t,\psi^u(t,x))=0\,,t\in[t_k,t_{k+1}),
      \end{equation}
$x\in\hbox{int} B(x^*,\g)\},\,\,0\leq k\leq N-1$, where $V(t,x)\mathop{=}\limits^{\hbox{not}} V^u(t,x;\l)\,,\,\,\l\in B(x^*,2\g)$. On the other hand, using integral equations \eqref{3.13}, we get
\begin{equation}\label{3.27}
   \psi^u(t,x)=V^u(t,x;\psi^u(t,x))\,,t\in[t_k,t_{k+1}),x\in B(x^*,\g)\},\,\,0\leq k\leq N-1.
      \end{equation}
By a direct derivation, from \eqref{3.27} we obtain
\begin{equation}\label{3.28}
\left\{
  \begin{array}{ll}
    \pl_t\psi^u(t,x)=[I_n-\pl_\l V^u(t,x;\psi^u(t,x))]^{-1}\pl_t V^u(t,x;\psi^u(t,x))\,,  \\
     \pl_x\psi^u(t,x)=[I_n-\pl_\l V^u(t,x;\psi^u(t,x))]^{-1}\pl_x V^u(t,x;\psi^u(t,x))\,,
  \end{array}
\right.
\end{equation}
for any $t\in[t_k,t_{k+1}),\,x\in\hbox{int} B(x^*,\g),\,0\leq k\leq N-1$. Here the nonsingular matrix used in the right-hand side of \eqref{3.28} relies on \eqref{2.22} in Lemma \ref{l:3} and \eqref{3.15} of Theorem \ref{t:2}. The equations \eqref{3.26} are valid for any $\l\in B(x^*,2\g)$ and, in particular for $\l=\psi^u(t,x)\in B(x^*,2\g)$, we get
\begin{equation}\label{3.29}
   \pl_t V(t,x;\psi^u(t,x))+\pl_x V^u(t,x,\psi^u(t,x))[g_0(x)+\mathop{\sum}\limits_{i=1}^mg_i(x)\pl_t\beta_i^u(t,\psi^u(t,x))]=0\,,
      \end{equation}
for any $t\in[t_k,t_{k+1}),\,x\in\hbox{int} B(x^*,\g),\,0\leq k\leq N-1$.
\end{proof}
Using \eqref{3.29} and multiplying the second equation in \eqref{3.28} by $[g_0(x)+\mathop{\sum}\limits_{i=1}^mg_i(x)\pl_t\beta_i^u(t,\psi^u(t,x))]$, we obtain the conclusion \eqref{3.25}. The proof is complete.

\begin{flushright}
\begin{tabular}{c}
  %\hline
 $1$. Abdus Salam School of Mathematical Science \\
  GC University, \\
  68 B, New MuslimTown, \\
  Lahore, Pakistan. \\
  54600 \\
  \\
  %\hline
 $2$. IMAR. Street Calea Grivit\"{u} nr.21B \\
  Bucharest, Romania. \\

\end{tabular}
\end{flushright}

\end{document}